\documentclass[a4paper, reqno]{amsart}
\usepackage{amssymb,amsmath,amsfonts,amsthm,mathrsfs}
\usepackage[colorlinks=true, citecolor=blue, anchorcolor=red]{hyperref}
\usepackage{enumerate}
\usepackage{tikz}

\newcommand{\R}{\mathbb R}

\newcommand{\C}{\mathbb C}

\renewcommand{\Re}{\mathop{\text{\upshape{Re}}}}

\newcommand{\sm}{\setminus}

\renewcommand{\AA}{\mathcal{A}}
\newcommand{\BB}{\mathcal{B}}

\newcommand{\HH}{\mathcal{H}}

\newcommand{\TT}{\mathcal{T}}

\renewcommand{\epsilon}{\varepsilon}
\renewcommand{\bar}[1]{\overline{#1}}

\renewcommand{\tilde}{\widetilde}
\renewcommand{\phi}{\varphi}
\renewcommand{\div}{\mathop{\text{\upshape{div}}}}

\newtheorem{theorem}{Theorem}[section]
\newtheorem{lemma}[theorem]{Lemma}

\newtheorem{proposition}[theorem]{Proposition}
\theoremstyle{definition}

\newtheorem{remark}[theorem]{Remark}

\numberwithin{equation}{section}

\allowdisplaybreaks

\begin{document}

\title[Exponential stability for a coupled system]
{Exponential stability for a coupled system of damped-undamped plate equations}
\author{Robert Denk}
\address{Universit\"at Konstanz, Fachbereich f\"ur Mathematik und Statistik,
         78457 Konstanz, Germany}
\email{robert.denk@uni-konstanz.de}
\author{Felix Kammerlander}
\address{Universit\"at Konstanz, Fachbereich f\"ur Mathematik und Statistik,
         78457 Konstanz, Germany}
         \email{felix.kammerlander@uni-konstanz.de}
\thanks{}
\date{February 2, 2017}
\begin{abstract}
We consider the transmission problem for a coupled system of undamped and structurally damped plate equations in two sufficiently smooth and bounded subdomains. It is shown that, independently of the size of the damped part, the damping is strong enough to produce uniform exponential decay of the  energy of the coupled system.
\end{abstract}

\subjclass[2010]{74K20; 74H40; 35B40; 35Q74}

\keywords{Plate equation, transmission problem, exponential stability}

\maketitle

\section{Introduction}

In this paper, we investigate a coupled system of linear plate equations where an undamped plate and a structurally damped plate are coupled through transmission conditions. From the point of view of applications, there is a connection to the suppression of vibration of elastic structures which is a main topic in material science. The undamped plate equation can be seen as a linear model for vibrating stiff objects where the potential energy is related to curvature-like terms, resulting in the bi-Laplacian operator as the main elastic operator, see, e.g., \cite{Leis86}, Chapter~12. For the purely undamped plate, we have no energy dissipation, and the governing semigroup is unitary. The model of structural damping is widely used to describe smoothing effects and loss of energy (cf. \cite{Russell84} for a discussion of the model). Here, we consider the damping term which has order two in the spatial variables, so it is of half order of the leading elastic term, see also \cite{Chen-Triggiani89} and \cite{DS15}  for the analysis of the structurally damped plate equation.

From a theoretical point of view, the resulting system can be seen as a transmission problem of mixed type: While the structurally damped plate equation is of parabolic nature, the undamped part is of dissipative nature. Below we will see that the damping is strong enough (independent of the size of the damped part) to obtain exponential stability for the semigroup of the coupled system. The analog result for a coupled system of thermoelastic / elastic plates was obtained in \cite{munoz_rivera-portillo_oquendo04}. The question of analyticity of the semigroup for a coupled thermoelastic plate / plate system is discussed in \cite{fernandez_sare-munoz_rivera11}. In \cite{Man13}, a plate / plate transmission problem with damping only on a part of the boundary with resulting polynomial decay was studied, see also \cite{Ammari-Vodev09} for the proof of exponential stability for a boundary stabilized plate / plate transmission problem. Transmission problems of plate / plate type can also be seen as an equation with  coefficients having jumps, cf. \cite{liu-williams00}.

In the system we consider the damping effect acting only through the transmission interface. Closely related is the question of boundary damping, see, e.g., \cite{Mustafa-Abusharkh15} or \cite{vila_bravo-munoz_rivera09}. In the literature, there are many results on coupled systems of plate / wave type (cf. \cite{ammari-nicaise10} and the references therein). In particular, in \cite{Avalos96} and \cite{Avalos-Lasiecka98}, the exponential stability for an abstract wave equation coupled with a plate-like equation on the boundary is studied. To our knowledge, the undamped / structurally damped plate system has not yet  been studied in  literature.

Let $\Omega\subset\R^n$ be a bounded domain with boundary $\Gamma_1 :=\partial \Omega$, and let $\Omega_2\subset\Omega$ be a non-empty bounded domain satisfying $\bar{\Omega_2}\subset\Omega$. We set $\Gamma:=\partial \Omega_2$ and $\Omega_1 := \Omega\setminus \bar{\Omega_2}$. Then, $\Gamma$ is the common interface (transmission interface) between $\Omega_1$ and $\Omega_2$, and $\partial \Omega_1 = \partial \Omega\cup\Gamma$ (see Figure~\ref{fig1} for the geometrical situation). All domains are assumed to be of class $C^4$. For technical reasons, we assume $n\le 4$, including the physically most relevant cases $n=1$ and $n=2$. Let $\nu$ denote the outer unit normal on $\Gamma_1$. On $\Gamma$, we choose $\nu$ to be the outer unit normal with respect to $\Omega_2$. Thus, $\nu$ is the inner unit normal vector on $\Gamma$ with respect to $\Omega_1$, see Figure~\ref{fig1}. Note that, apart from the smoothness, we do not impose a geometrical condition on the domains.

\begin{figure}[ht]
\begin{center}
\begin {tikzpicture}[x=3em,y=3em]
\draw (0,0) ellipse (1.05  and 0.7);
\draw (0.5, 0.2 ) ellipse (2.12 and 1.58);
\draw [->,bend left=15]  (-2.2, 0.5) to (-1.6,0.35);
\draw [->,bend left=15] (1.3, -0.5) to (0.95,-0.3);
\draw [->, thick] (0.77,0.47) to node{$\qquad\nu$} +(0.5,0.6);
\draw [->, thick] (2.58,0.5) to node[below]{$\nu$} +(0.8,0.2);
\node at (0,0) {$\Omega_2$};
\node at (-0.3, 1) {$\Omega_1$};
\node at (-2.5,0.5) {$\Gamma_1$};
\node at (1.5,-0.5) {$\Gamma$};
\end{tikzpicture}
\end{center}
\caption{The set $\Omega=\Omega_1\cup\Gamma\cup\Omega_2$. \label{fig1}}
\end{figure}
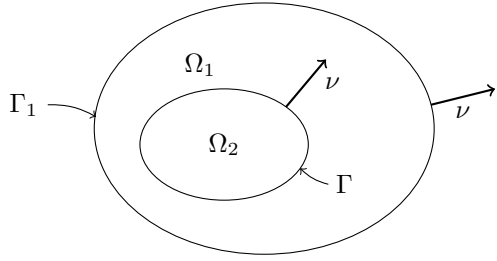

We consider a transmission problem for thin plates where the plate in $\Omega_2$ is undamped and the material in $\Omega_1$ is structurally damped. More precisely, we are
looking for solutions $u_i\colon \Omega_i \to \C$ of the system \index{Transmission problem!hyperbolic-parabolic!partially structurally damped}
\begin{align}
	\partial_t^2 u_1 + \Delta^2 u_1 - \rho \Delta \partial_t u_1	& = 0 \quad		 \text{ in } (0, \infty) \times \Omega_1,																																\label{2_eq_tm_damped} \\
	\partial_t^2 u_2 + \Delta^2 u_2									& = 0 \quad		 \text{ in } (0, \infty) \times \Omega_2	\label{2_eq_tm_undamped}
\end{align}
with {\it clamped boundary conditions}
\begin{align}
	u_1 = \partial_\nu u_1 = 0 \quad \text{ on } \Gamma_1																		\label{2_eq_tm_clamped}
\end{align}
Here, $\rho \in (0,\infty)$ is the damping factor.
The transmission conditions on $\Gamma$ are given by
\begin{align}
	u_1 														& = u_2, 														\label{2_eq_tm_0}\\
	\partial_\nu u_1											& = \partial_\nu u_2, 											\label{2_eq_tm_1}\\
	\Delta u_1													& = \Delta u_2, 												\label{2_eq_tm_2}\\
	-\rho \partial_\nu \partial_t u_1 + \partial_\nu \Delta u_1	& = \partial_\nu \Delta u_2.										\label{2_eq_tm_3}
\end{align}
The problem is completed by the initial conditions
\begin{align}
	u_1(0, \cdot) = u_1^0, & \quad \partial_t u_1(0, \cdot) = u_1^1 \quad \text{ in } \Omega_1,					\label{2_eq_tm_init_1}\\
	u_2(0, \cdot) = u_2^0, & \quad \partial_t u_2(0, \cdot) = u_2^1 \quad \text{ in } \Omega_2.					\label{2_eq_tm_init_2}
\end{align}
The \index{Energy!partially structurally damped transmission problem} energy of the system \eqref{2_eq_tm_damped}-\eqref{2_eq_tm_init_2} is defined as
\begin{equation}\label{2_eq_tm_energy}
	E(t) := \frac12 \int_{\Omega_1} \vert \partial_t u_1 \vert^2 + \vert \Delta u_1 \vert^2 \,dx
			+ \frac12 \int_{\Omega_2} \vert \partial_t u_2 \vert^2 + \vert \Delta u_2 \vert^2 \,dx.
\end{equation}
If $(u_1, u_2)$ is a solution, integration by parts yields the estimate
\begin{align}
	\frac{d}{dt} E(t)	& = -\rho \Vert \nabla \partial_t u_1 \Vert_{L^2(\Omega_1)}^2 \, \leq 0.	\label{2_eq_total_energy}
\end{align}
Note that $u_i = \partial_\nu u_i = 0$ on $\Gamma_i$ implies $\partial_t u_i = \partial_\nu \partial_t u_i = 0$ on $\Gamma_i$ for $i = 1, 2.$
The estimate shows that the energy of the transmission problem is decreasing in time and the dissipation is caused by the damped part $u_1.$

Our main result, Theorem~\ref{thm_exp_stability_structurally_damped_undamped} below, states that the damping in $\Omega_1$ is strong enough to achieve exponential decrease of the energy, i.e. there exist constants
$C, \kappa > 0$ such that
\[
	E(t) \leq CE(0)e^{-\kappa t}
\]
holds for all $t \geq 0.$ To prove this, we first study the resolvent and the spectrum of the first-order system related to \eqref{2_eq_tm_damped}--\eqref{2_eq_tm_3} in Section~2. In the proof of exponential stability, we also need an a priori estimate on the damped part which is obtained in Section~3 with the help of the interpolation-extrapolation scales of Banach spaces. Finally, the results from Section~2 and 3 are used to prove the main result on exponential stability in Section~4.

%%%%%%%%%%%%%%%%%%%%%%%%%%%%%%%%%%%%%%%%%%%%%%%%%%%%%%%%%%%% SUBSECTION - WELL-POSEDNESS %%%%%%%%%%%%%%%%%%%%%%%%%%%%%%%%%%%%%%%%%%%%%%%%%%%%%%%%%%%%%%%%%%%%%%%%
\section{The spectrum of the first-order system}\label{subsec:2_well_posedness}
%%%%%%%%%%%%%%%%%%%%%%%%%%%%%%%%%%%%%%%%%%%%%%%%%%%%%%%%%%%%%%%%%%%%%%%%%%%%%%%%%%%%%%%%%%%%%%%%%%%%%%%%%%%%%%%%%%%%%%%%%%%%%%%%%%%%%%%%%%%%%%%%%%%%%%%%%%%%%%%%%
Setting $U:= (u_1,u_2,v_1,v_2)^\top$ with $v_j:=\partial_t u_j$, we rewrite the transmission problem \eqref{2_eq_tm_damped}-\eqref{2_eq_tm_init_2} as
\begin{equation}\label{2-1}
	\partial_t U (t)  - \AA U (t)= 0\; (t>0), \quad U(0)=U_0
\end{equation}
where the operator $\AA$ acts in form of the matrix
\[ A(D) := \begin{pmatrix}
  0 & 0 & 1 & 0 \\
  0 & 0 & 0 & 1 \\
  -\Delta^2 & 0 & \rho\Delta & 0\\
  0 & -\Delta^2 & 0 & 0
\end{pmatrix}.\]
As the basic space for the first two components $(u_1,u_2)$, we will choose
\begin{align*}
	X(\Omega) := \big\{ (u_1, u_2) \in H^2(\Omega_1) \times H^2(\Omega_2) :\ & u_1 = \partial_\nu u_1 = 0 \text{ on } \Gamma_1, \\
		& u_1 = u_2 \text{ on } \Gamma, \, \partial_\nu u_1 = \partial_\nu u_2 \text{ on } \Gamma \big\}.
\end{align*}

\begin{remark}
  \label{2.1}
  a) Let $(u_1,u_2)\in H^2(\Omega_1)\times H^2(\Omega_2)$. Then the conditions $u_1=u_2$, $\partial_\nu u_1=\partial_\nu u_2$ on $\Gamma$ are equivalent to $u:= \chi_{\Omega_1} u_1 + \chi_{\Omega_2} u_2 \in H^2(\Omega)$, where $\chi_{\Omega_j}$ stands for the characteristic function of $\Omega_j$, i.e. $\chi_{\Omega_j}(x)=1$ for $x\in \Omega_j$ and $\chi_{\Omega_j}(x)=0$ else. Therefore, we have
  \[ X(\Omega) = \{ (u|_{\Omega_1}, u|_{\Omega_2}): u\in H^2_0(\Omega)\}.\]
  In the following, we will several times use the identification of $(u_1,u_2)$ and $u$.

  b) The norm in $X(\Omega)$ is defined as
  \[ \|(u_1,u_2)\|_{X(\Omega)} := \Big( \|\Delta u_1\|_{L^2(\Omega_1)}^2 + \|\Delta u_2\|_{L^2(\Omega_2)}^2\Big)^{1/2}.\]
  Note that this norm is equivalent to the standard norm $(\|u_1\|_{H^2(\Omega_1)}^2 + \|u_2\|_{H^2(\Omega_2)}^2)^{1/2}$. In fact, due to the invertibility of the Dirichlet Laplacian in $\Omega$, the norms $\|\Delta u\|_{L^2(\Omega)}$ and $\|u\|_{H^2(\Omega)}$ are equivalent on the space $H^2(\Omega)\cap H^1_0(\Omega)$. As $H^2_0(\Omega)$ is a closed subspace of $H^2(\Omega)\cap H^1_0(\Omega)$, these norms are also equivalent on $H^2_0(\Omega)$, and now the assertion follows from part a) (see also \cite{Has14}, Proposition 2.1 and Proposition 2.2).
\end{remark}

We say that the transmission conditions \eqref{2_eq_tm_2} and \eqref{2_eq_tm_3} are weakly satisfied if
\begin{align}
	\langle \Delta^2 u_1 & - \rho \Delta v_1, \phi_1 \rangle_{L^2(\Omega_1)}
		+ \langle \Delta^2 u_2, \phi_2 \rangle_{L^2(\Omega_2)} \notag \\
	& = \langle \Delta u_1, \Delta \phi_1 \rangle_{L^2(\Omega_1)} + \langle \Delta u_2, \Delta \phi_2 \rangle_{L^2(\Omega_2)}
		+ \rho \langle \nabla v_1, \nabla \phi_1 \rangle_{L^2(\Omega_1)}
\label{eq_tm_weak}
\end{align}
holds for all $(\phi_1, \phi_2) \in X(\Omega)$. Let
\[ \HH := X(\Omega) \times L^2(\Omega_1) \times L^2(\Omega_2).\]
Then we define the operator $\AA\colon \HH\supset D(\AA)\to \HH$ by

\begin{align*}
	D(\AA) & := \big\{ (u_1, u_2, v_1, v_2) \in  X(\Omega)\times X(\Omega):  \Delta^2 u_1 \in L^2(\Omega_1),\,  \Delta^2 u_2\in L^2(\Omega_2),\\
&\qquad   \eqref{2_eq_tm_2} \text{ and } \eqref{2_eq_tm_3} \text{ are weakly satisfied} \big\}
\end{align*}
and $\AA U := A(D) U \;(U\in D(\AA))$.

We will see in Lemma~\ref{lem_strong_solution} below that the functions in $D(\AA)$ are sufficiently smooth and the transmission conditions hold  in the sense of traces.

\begin{theorem}\label{2_thm_well_posedness_damped_undamped}
	The operator $\AA$ is the generator of a $C_0$-semigroup of contractions on the Hilbert space $\HH.$ Therefore, for all
	$U_0 \in D(\AA)$ the Cauchy problem \eqref{2-1}
	has a unique classical solution $U \in C^1([0, \infty), \HH)$ with $U(t) \in D(\AA)$ for all $t \geq 0.$
\end{theorem}

\begin{proof}
By the definition of $D(\AA)$ and the weak transmission conditions \eqref{eq_tm_weak}, it is immediately seen that
\[ \Re\,\langle \AA U,U\rangle_{\HH}  = - \rho \Vert \nabla v_1 \Vert_{L^2(\Omega_1)}^2\quad (U\in D(\AA)).\]
Hence, $\AA$ is dissipative. We want to show that $1-A$ is surjective. For this, let $F = (f_1, f_2, g_1, g_2)^\top \in \HH.$ We have to  find
	$U = (u_1, u_2, v_1, v_2)^\top \in D(\AA)$ satisfying
	\begin{align*}
		u_1 - v_1 								& = f_1, \\
		u_2 -v_2 								& = f_2, \\
		v_1 + \Delta^2 u_1 - \rho \Delta v_1	& = g_1, \\
		v_2 + \Delta^2 u_2						& = g_2.
	\end{align*}	
	Plugging in $v_i = u_i - f_i$ for $i = 1, 2$ in the third and fourth equation yields that we have to solve
\begin{align}
u_1 + \Delta^2 u_1 - \rho \Delta u_1	& = g_1 + f_1 - \rho \Delta f_1, \label{2-2} \\
u_2 + \Delta^2 u_2						& = g_2 + f_2\label{2-3}
\end{align}
as equalities in $L^2(\Omega_1)$ and $L^2(\Omega_2),$ respectively.

 We define the continuous sesquilinear form $B \colon X(\Omega) \times X(\Omega) \to \C$ by
	\begin{align*}
		B(u, \phi) 	& = \langle u_1, \phi_1 \rangle_{L^2(\Omega_1)} + \langle \Delta u_1, \Delta \phi_1 \rangle_{L^2(\Omega_2)}
						+ \rho \langle \nabla u_1, \nabla \phi_1 \rangle_{L^2(\Omega_1)} \\
					& \quad + \langle u_2, \phi_2 \rangle_{L^2(\Omega_2)} + \langle \Delta u_2, \Delta \phi_2 \rangle_{L^2(\Omega_2)}	
	\end{align*}
	for $u = (u_1, u_2), \phi = (\phi_1, \phi_2) \in X(\Omega)$.  Since
	\[
		\Re B(u, u) \geq \Vert (u_1, u_2) \Vert_{X(\Omega)}^2 \qquad (u \in X(\Omega)),
	\]
	$B$ is coercive. Obviously, the mapping $\Lambda \colon X(\Omega) \to \C$ defined by
	\begin{align*}
		\Lambda(\phi) :=
		\langle g_1 + f_1, \phi_1 \rangle_{L^2(\Omega_1)} + \rho \langle \nabla f_1, \nabla\phi_1 \rangle_{L^2(\Omega_1)} + \langle g_2 + f_2, \phi_2 \rangle_{L^2(\Omega_2)}
	\end{align*}
	for $\phi = (\phi_1, \phi_2) \in X(\Omega)$ is linear and continuous. By the theorem of Lax-Milgram, there exists a unique $u = (u_1, u_2) \in X(\Omega)$ such that
	$B(u, \phi) = \Lambda(\phi)$ holds for all $\phi \in X(\Omega)$. In particular, choosing $(\phi_1, \phi_2) \in C_0^\infty(\Omega_1) \times C_0^\infty(\Omega_2) \subset X(\Omega)$,
	we see that \eqref{2-2} and \eqref{2-3} hold in the sense of distributions in $\Omega_1$ and $\Omega_2$, respectively. As the right-hand side of \eqref{2-2} belongs to $L^2(\Omega_1)$, the same holds for the left-hand side. Due to $u_1\in H^2(\Omega_1)$, this yields $\Delta^2 u_1\in L^2(\Omega_1)$. In the same way, we see that \eqref{2-3} holds as equality in $L^2(\Omega_2)$ and that $\Delta^2 u_2\in L^2(\Omega_2)$.

Set $v_1 := u_1-f_1$ and $v_2:= u_2-f_2$. By \eqref{2-2}--\eqref{2-3}, we have
\begin{equation}\label{2-4}
\begin{aligned}
  \Delta^2 u_1 - \rho\Delta v_1 & = - u_1 + g_1 + f_1,\\
  \Delta^2 u_2 & = - u_2 + g_2 + f_2.
\end{aligned}
\end{equation}
Let  $\phi=(\phi_1,\phi_2)\in X(\Omega)$. Then, because of \eqref{2-4} and $B(u, \phi) = \Lambda(\phi)$, we get
\begin{align*}
 \langle \Delta^2 u_1 & - \rho \Delta v_1, \phi_1 \rangle_{L^2(\Omega_1)}
		+ \langle \Delta^2 u_2, \phi_2 \rangle_{L^2(\Omega_2)} \\
& = \langle - u_1 + g_1 + f_1, \phi_1\rangle_{L^2(\Omega_1)} + \langle - u_2 + g_2 + f_2, \phi_2\rangle_{L^2(\Omega_2)} \\
& =  \langle \Delta u_1, \Delta \phi_1 \rangle_{L^2(\Omega_1)} + \langle \Delta u_2, \Delta \phi_2 \rangle_{L^2(\Omega_2)}
			+ \rho \langle \nabla v_1, \nabla \phi_1 \rangle_{L^2(\Omega_1)}.
\end{align*}
Therefore, the weak transmission conditions \eqref{eq_tm_weak} are satisfied. Altogether, we have seen that $U:= (u_1,u_2,v_1,v_2)^\top $ belongs to $D(\AA)$. Because of \eqref{2-2}--\eqref{2-3} and the definition of $v_1,v_2$, we also have $(1-\AA)U=F$. Therefore, $1-\AA$ is surjective which implies that $\AA$ is densely defined (see \cite{Paz12}, Theorem 4.6). An application of the Lumer-Phillips theorem now yields the statement of the theorem.
\end{proof}

\begin{remark}\label{rem_0_resolvent}
%	\begin{enumerate}[a)]
%		\item
In the same way as in the previous proof, one can show that the operator $\AA$ is continuously invertible, i.e. $0$ belongs to the resolvent set $\rho(\AA)$. To show this, we now have to consider
\begin{align}
  	\Delta^2 u_1		& = g_1 - \rho \Delta f_1, \label{2-5}\\
	\Delta^2 u_2		& = g_2 \label{2-6}
\end{align}
instead of \eqref{2-2}--\eqref{2-3}. The sesquilinear form $B$ and the functional $\Lambda$ are now defined by  $B(u, \phi) = \langle u, \phi \rangle_{X(\Omega)}$ and
			\begin{align*}
				\Lambda(\phi) :=
				\langle g_1, \phi_1 \rangle_{L^2(\Omega_1)} + \rho \langle \nabla f_1, \nabla\phi_1 \rangle_{L^2(\Omega_1)} + \langle g_2, \phi_2 \rangle_{L^2(\Omega_2)}
			\end{align*}
			for $u = (u_1, u_2), \phi = (\phi_1, \phi_2) \in X(\Omega)$.

As before, we see that there exists a unique solution $u = (u_1, u_2) \in X(\Omega)$ satisfying $B(u, \phi) = \Lambda(\phi)$ for all $\phi \in X(\Omega)$. Moreover, setting $v_j := -f_j$, the vector $U:=(u_1,u_2,v_1,v_2)^\top $ belongs to $D(\AA)$ and satisfies $-\AA U = F$.

On the other hand, if $\tilde U\in D(\AA)$ solves $-\AA \tilde U = F$, then $B(\tilde u,\phi)=\Lambda(\phi)$ holds for all $\phi\in X(\Omega)$ due to the definition of $D(\AA)$ and the weak transmission conditions. Therefore, $U=\tilde U$, and $\AA\colon D(\AA)\to \HH$ is a bijection. Since $\AA$ is the generator of a $C_0$-semigroup, $\AA$ is closed and the continuity of $\AA^{-1}\colon \HH \to \HH$ follows. Therefore, $0\in \rho(A)$.	
\end{remark}

\begin{lemma}\label{lem_strong_solution}
a) The domain of $\AA$ is given by
\begin{equation}
  \label{2-7}
\begin{aligned}
		D(\AA) & = \big\{ (u_1, u_2, v_1, v_2) \in \big(H^4(\Omega_1)\times H^4(\Omega_2)\big)\cap X(\Omega) \times X(\Omega): \\
& \qquad \Delta u_1 = \Delta u_2\text{ on } \Gamma, \, -\rho \partial_\nu v_1 + \partial_\nu \Delta u_1 = \partial_\nu \Delta u_2 \text{ on } \Gamma\big\}.
\end{aligned}
\end{equation}
Here, the equalities on $\Gamma$ can be understood as equalities in the trace spaces $H^{3/2}(\Gamma)$ and $H^{1/2}(\Gamma)$, respectively.

b) The operator $\AA$ has compact resolvent and, consequently, discrete spectrum.
\end{lemma}

\begin{proof}
a) Let $\tilde U\in D(\AA)$ and $F=(f_1,f_2,g_1,g_2)^\top:= -\AA \tilde U\in \HH$. To show the statement, we construct a strong solution $U$ of $-\AA U = F$ belonging to the right-hand side of \eqref{2-7} and show that $U=\tilde U$. So we consider
	\begin{align}
		\Delta^2 u_1 & = g_1 - \rho \Delta f_1,		\label{eq_tm_strong_1} \\
		\Delta^2 u_2 & = g_2 \label{eq_tm_strong_2}
	\end{align}
	in $L^2(\Omega_1) \times L^2(\Omega_2)$ with boundary conditions
	\begin{align}
		u_1 = \partial_\nu u_1 = 0 \text{ on } \Gamma_1 \label{eq_tm_strong_3}
	\end{align}
and transmission conditions
 	\begin{align}
		u_1 - u_2 & = 0,  \label{eq_tm_strong_4}\\
		\partial_\nu u_1 - \partial_\nu u_2 & = 0, \label{eq_tm_strong_5}\\
		\partial_\nu^2 u_1 - \partial_\nu^2 u_2 & = 0,  \label{eq_tm_strong_6}\\
		\partial_\nu^3 u_1 - \partial_\nu^3 u_2 & = -\rho \partial_\nu f_1 \label{eq_tm_strong_7}.
	\end{align}
Concerning the higher-order transmission conditions \eqref{eq_tm_strong_6} and \eqref{eq_tm_strong_7}, note that  for all $(u_1,u_2)\in H^4(\Omega_1)\times H^4(\Omega_2)$ satisfying \eqref{eq_tm_strong_4} and \eqref{eq_tm_strong_5}  all tangential derivatives of $u_1-u_2$ and $\partial_\nu u_1-\partial_\nu u_2$ along $\Gamma$ disappear. Therefore, for such $u$ the transmission conditions \eqref{eq_tm_strong_6}--\eqref{eq_tm_strong_7} are equivalent to the conditions
	\begin{align*}
		\Delta u_1 - \Delta u_2 & = 0, \\
		\partial_\nu \Delta u_1 - \partial_\nu \Delta u_2 & = -\rho \partial_\nu f_1.
	\end{align*}

Define the operator $\BB\colon L^2(\Omega)\supset D(\BB)\to L^2(\Omega)$ by $D(\BB):= H^4(\Omega)\cap H^2_0(\Omega)$ and $\BB w := \Delta^2 w$. Then, $\BB$ is a selfadjoint operator with $0\in\rho(\BB)$. To  construct a strong solution of the transmission problem \eqref{eq_tm_strong_1}--\eqref{eq_tm_strong_7}, we first eliminate the inhomogeneity on the right-hand side of \eqref{eq_tm_strong_7}.  By \cite{Tr78}, Section 4.7.1, p. 330, the mapping
	\begin{align*}
		\mathscr{R} h :=
		\left( h\vert_{\partial \Omega_1}, \partial_\nu h\vert_{\partial \Omega_1}, \partial_\nu^2 h\vert_{\partial \Omega_1}, \partial_\nu^3 h\vert_{\partial \Omega_1} \right)^\top
	\end{align*}
	is a retraction from $H^4(\Omega_1)$ onto $\prod_{j=0}^3 H^{4-1/2-j}(\partial \Omega_1).$ Therefore, there exists a function $h \in H^4(\Omega_1)$ such that
	\begin{align*}
		\mathscr{R} h = (0, 0, 0, -\chi_\Gamma \rho \partial_\nu f_1)^\top.
	\end{align*}
Here again $\chi_\Gamma$ stands for the characteristic function of $\Gamma$.
	We define $w := \BB^{-1}G \in H^4(\Omega) \cap H^2_0(\Omega),$ where
	\begin{align*}
		G = \chi_{\Omega_1} (g_1 - \rho \Delta f_1 - \Delta^2 h) + \chi_{\Omega_2} g_2 \in L^2(\Omega).
	\end{align*}
Finally, we set $u_1 := w\vert_{\Omega_1} + h$ and $u_2 := w\vert_{\Omega_2}.$ Then, $u = (u_1, u_2) \in H^4(\Omega_1) \times H^4(\Omega_2)$ satisfies
	the strong transmission problem \eqref{eq_tm_strong_1}--\eqref{eq_tm_strong_7}.
Therefore, $U:= (u_1,u_2,v_1,v_2)^\top$ with $v_j:=-f_j$ belongs to the right-hand side of \eqref{2-7} and solves $-A(D) U = F$.

	On the other hand, 	using integration by parts and the fact that $u$ solves the strong transmission problem, we see that $U$ satisfies the weak transmission conditions \eqref{eq_tm_weak}. Therefore, $U$ belongs to $D(\AA)$ and solves $-\AA U = F$. By Remark~\ref{rem_0_resolvent}, this solution is unique which implies $U=\tilde U$.

b) Due to a), we have
\[ D(\AA) \subset \big(H^4(\Omega_1)\times H^4(\Omega_2)\big)\cap X(\Omega) \times X(\Omega).\]
By the Rellich-Kondrachov theorem, the space on the right-hand side is compactly embedded into $\HH$. Therefore, $\AA^{-1}$ is compact, and the spectrum of $\AA$ is discrete.
\end{proof}

We already know that the spectrum of $\AA$ is discrete and that $0$ is no eigenvalue. In fact, there are no purely imaginary eigenvalues of $\AA$, as the next result shows.

\begin{theorem}
\label{2_prop_imaginary_resolvent}
	The imaginary axis is a subset of the resolvent set of $\AA,$ i.e. $\mathrm{i}\R \subset \rho(\AA).$
\end{theorem}

\begin{proof}
Assume that $U = (u_1, u_2, v_1, v_2)^\top \in D(\AA)$ satisfies $(-\mathrm{i}\lambda + \AA)U = 0$ with $\lambda\in\R\setminus\{0\}$. Then $v_j=\mathrm i \lambda u_j$ for $j=1,2$, and  $(u_1, u_2)$ satisfies
	\begin{align}
		-\Delta^2 u_1 + \mathrm{i}\lambda\rho \Delta u_1 + \lambda^2 u_1	& = 0 \quad \text{ in } \Omega_1,	\label{2_prop_ir_eq_1} \\
		-\Delta^2 u_2 + \lambda^2 u_2										& = 0 \quad \text{ in } \Omega_2	\label{2_prop_ir_eq_2}
	\end{align}
	with boundary conditions $u_1 = \partial_\nu u_1 = 0$ on $\Gamma_1$ and transmission conditions
	\begin{align*}
		u_1 																		& = u_2, \\
		\partial_\nu u_1															& = \partial_\nu u_2, \\
		\Delta u_1																	& = \Delta u_2, \\
		-\mathrm{i}\lambda\rho \partial_\nu u_1 + \partial_\nu \Delta u_1			& = \partial_\nu \Delta u_2
	\end{align*}
	on the common interface $\Gamma.$ \\
	We will show that $(u_1, u_2) = 0.$
	We multiply \eqref{2_prop_ir_eq_1} and \eqref{2_prop_ir_eq_2} with $\overline{u_1}$ and $\overline{u_2},$ respectively. Summing up and performing an integration
	by parts yields
	\begin{align*}
		- \Vert \Delta u_1 \Vert^2_{L^2(\Omega_1)} - \mathrm{i}\lambda\rho \Vert \nabla u_1 \Vert^2_{L^2(\Omega_1)} + \lambda^2 \Vert u_1 \Vert^2_{L^2(\Omega_1)} & \\
		- \Vert \Delta u_2 \Vert^2_{L^2(\Omega_2)} + \lambda^2 \Vert u_2 \Vert^2_{L^2(\Omega_2)} & = 0.
	\end{align*}
	Here we have used the boundary conditions as well as the transmission conditions on $\Gamma.$
	Considering only the imaginary part we get $\Vert \nabla u_1 \Vert_{L^2(\Omega_1)} = 0.$
	Together with $u_1\vert_{\Gamma_1} = 0$ we obtain $u_1 = 0.$
	Therefore, $u_2$ satisfies the boundary value problem
	\begin{align}
			-\Delta^2 u_2 + \lambda^2 u_2										& = 0 \quad \text{ in } \Omega_2, \label{2-8}\\
			u_2 = \partial_\nu u_2 = \Delta u_2	= \partial_\nu \Delta u_2	& = 0 \quad \text{ on } \Gamma = \partial\Omega_2.\label{2-9}
	\end{align}
Because of \eqref{2-9}, the trivial extension $\tilde u_2$ by zero to $\R^n$ belongs to $H^4(\R^n)$ and satisfies $\Delta^2 \tilde u_2 = \lambda^2 \tilde u_2$ in $\R^n$. As $\Delta^2$ in $L^2(\R^n)$ has no eigenvalues, this implies $\tilde u_2=0$ and therefore $u_2=0$. Altogether we have seen $U=0$.
\end{proof}

The last results already implies strong stability of the semigroup $(\TT(t))_{t\ge 0}$ generated by $\AA$, i.e., for any $U_0\in \HH$ we have $\|\TT(t)U_0\|_\HH \to 0 \;(t\to\infty)$ (see \cite{AB88}, Theorem~2.4). We will see in Section~4 that $\TT$ is even exponentially stable.

\section{A priori estimates for the damped plate equation}

For the proof of exponential stability of the coupled damped--undamped plate equation, we need some a priori estimates for the damped part. For this, we will apply the theory of interpolation-extrapolation scales due to Amann (see \cite{Am95}, Chapter V).

Throughout this section, let $U\subset\R^n$ be a bounded $C^4$-domain. We define the operator $A$ in the space $H_0^2(U)\times L^2(U)$ by
\begin{equation}\label{3-6}
\begin{aligned}
 D(A) & := (H^4(U)\cap H^2_0(U))\times H_0^2(U),\\
 A & := \begin{pmatrix}
   0 & 1 \\ -\Delta^2 & \rho\Delta
 \end{pmatrix}.
\end{aligned}
\end{equation}
It was shown in \cite{Chen-Triggiani89}, Proposition~3.1 (see also \cite{DS15}, Theorem~5.1) that $A$ generates an analytic exponentially stable $C_0$-semigroup in $H_0^2(U)\times L^2(U)$. To extrapolate this result to spaces of negative regularity, we need to determine the adjoint operator $A'$ considered in the dual spaces. In the following, $\langle \cdot,\cdot\rangle_{X'\times X}$ denotes the dual pairing in a Banach space $X$. We begin with a small observation on the bi-Laplacian operator.

\begin{remark}
  \label{3.1}
  Under the above assumptions on $U$, the operator $\Delta^2\colon H_0^2(U)\to H^{-2}(U)$ is an isomorphism. In fact, we have the coercive estimate
  \[ \langle \Delta^2u,u\rangle_{H^{-2}(U)\times H_0^2(U)} = \|\Delta u\|_{L^2(U)}^2 \ge C\|u\|_{H^2(U)}^2\quad (u\in H_0^2(U)).\]
  Here the last inequality holds by elliptic regularity and invertibility of the Dirichlet Laplacian $\Delta_D\colon H^2(U)\cap H_0^1(U)\to L^2(U)$. Now an application of the Lax-Milgram theorem yields the invertibility of $\Delta^2\colon H_0^2(U)\to H^{-2}(U)$.
\end{remark}

\begin{lemma}
\label{3.2}
	The adjoint operator $A'$ of $A$ is given by
	\begin{align*}
		A' \colon H^{-2}(U) \times L^2(U) \supset D(A')	& := L^2(U) \times H^2_0(U) \to H^{-2}(U) \times L^2(U), \\
		A' & := \begin{pmatrix} 0 & -\Delta^2 \\ 1 & \rho \Delta \end{pmatrix}.
	\end{align*}
\end{lemma}

\begin{proof}
	We define $E := H^2_0(U) \times L^2(U)$ and $\widetilde D := L^2(U) \times H^2_0(U) \subset E',$
	where $E' = H^{-2}(U) \times L^2(U).$ Then, for all
	$v = (v_1, v_2) \in \widetilde D$ and $$u = (u_1, u_2) \in D(A) = \left( H^4(U) \cap H^2_0(U) \right) \times H^2_0(U),$$ integration by parts and the definition of
	distributional derivatives yield
	\begin{align*}
		v(Au)
			& = \langle v_1, u_2 \rangle_{L^2(U)} + \langle v_2, -\Delta^2 u_1 \rangle_{L^2(U)} + \langle v_2, \rho \Delta u_2 \rangle_{L^2(U)} \\
			& = \langle v_1, u_2 \rangle_{L^2(U)} + \langle -\Delta v_2, \Delta u_1 \rangle_{L^2(U)} + \langle \rho \Delta v_2, u_2 \rangle_{L^2(U)} \\
			& = \langle -\Delta^2 v_2, u_1 \rangle_{H^{-2}(U) \times H^2_0(U)} + \langle v_1, u_2 \rangle_{L^2(U)} + \langle \rho \Delta v_2, u_2 \rangle_{L^2(U)} \\
			& = w_1(u_1) + w_2(u_2),
	\end{align*}
	with $w_1 := -\Delta^2 v_2 \in H^{-2}(U)$ and $w_2 := v_1 + \rho \Delta v_2 \in L^2(U).$
	Therefore, we set
	\begin{align*}
		 \widetilde A := \begin{pmatrix} 0 & -\Delta^2 \\ 1 & \rho \Delta \end{pmatrix}
	\end{align*}
	with $D(\widetilde A) := \widetilde D.$ With this definition, we have $v(Au) = (\widetilde A v)(u)$ for all $u \in D(A)$ and all $v \in \widetilde D.$
	Moreover, for all $v \in \widetilde D$ the mapping
	$[u \mapsto v(Au)] \colon D(A) \to \C$ is continuous with respect to $\Vert \cdot \Vert_E.$ Hence, we have $\widetilde A \subset A'.$ \\
	
	Let $u \in D(A)$ and $v \in E'.$ Then
	\begin{align}
		v(Au)
			& = \langle v_1, u_2 \rangle_{H^{-2}(U) \times H^2_0(U)} + \langle v_2, -\Delta^2 u_1 \rangle_{L^2(U)} + \langle v_2, \rho \Delta u_2 \rangle_{L^2(U)}.
		\label{eq_vAu}
	\end{align}
	Now, let $v \in D(A').$ Then, the mapping $[u \mapsto v(Au)] \colon D(A) \to \C$ can be extended to a linear, continuous mapping from $E$ to $\C.$
	In particular, considering
	\begin{align*}
		\vert \langle v_2, \Delta^2 u_1 \rangle_{L^2(U)} \vert = \vert v(A (u_1, 0)) \vert \leq C \Vert (u_1, 0) \Vert_E = C \Vert u_1 \Vert_{H^2(U)}
	\end{align*}
	for $u_1 \in H^4(U) \cap H^2_0(U),$		
	it holds that
	\begin{align}
		\varphi \colon H^4(U) \cap H^2_0(U) \to \C, \quad u_1 \mapsto \varphi(u_1) := \langle v_2, \Delta^2 u_1 \rangle_{L^2(U)}
		\label{eq_phi_H2}
	\end{align}
	is continuous with respect to $\Vert \cdot \Vert_{H^2(U)}.$ By Remark~\ref{3.1},
	\begin{align}
		\Delta^2 \colon H^2_0(U) \to H^{-2}(U)	
		\label{eq_Delta2_Iso_H2}
	\end{align}	
	is an isomorphism. Therefore, \eqref{eq_phi_H2} and \eqref{eq_Delta2_Iso_H2} imply that
	\begin{align*}
		\left[ \widetilde u_1 \mapsto \varphi \left( (\Delta^2)^{-1} \widetilde u_1 \right) = \langle v_2, \widetilde u_1 \rangle_{L^2(U)} \right]
		\colon L^2(U) \to \C
	\end{align*}
	is continuous considered as a mapping from $(L^2(U), \Vert \cdot \Vert_{H^{-2}(U)})$ to $\C.$
	By the density of $L^2(U) \subset H^{-2}(U)$, there exists a unique continuous
	extension $$\widetilde \varphi \in \left( H^{-2}(U) \right)' = H^2_0(U)$$ of this mapping. Together with
	\begin{align*}
		\langle \widetilde \varphi, \widetilde u_1 \rangle_{H^2_0(U) \times H^{-2}(U) } = \langle v_2, \widetilde u_1 \rangle_{H^2_0(U) \times H^{-2}(U)}
	\end{align*}
	for $ \widetilde u_1 \in L^2(U),$ we deduce $v_2 = \widetilde \varphi \in H^2_0(U).$ \\
	The fact that $v_2 \in H^2_0(U)$ implies that the last term in \eqref{eq_vAu},
	\begin{align*}
		\left[ u_2 \mapsto \langle v_2, \rho \Delta u_2 \rangle_{L^2(U)} = \langle v_2, \rho \Delta u_2 \rangle_{H^2_0(U) \times H^{-2}(U)} \right]
		\colon H^2_0(U) \to \C,
	\end{align*}
	is continuous on $L^2(U)$. Since \eqref{eq_vAu} needs to be continuous, by setting $u_1 = 0$ it follows that also the first term
	\begin{align*}
		\left[ u_2 \mapsto \langle v_1, u_2 \rangle_{H^{-2}(U) \times H^2_0(U)} \right] \colon H^2_0(U) \to \C
	\end{align*}
	can be extended continuously to $L^2(U),$ which means $v_1 \in L^2(U).$ \\
	We have shown that $v \in D(A')$ implies $v_2 \in H^2_0(U)$ and $v_1 \in L^2(U),$ i.e. $v \in \widetilde D.$ Hence, we obtain
	$\widetilde D = D(A')$ and therefore $\widetilde A = A'.$
\end{proof}

In the following, $$\Sigma_{\phi} = \{ z \in \C \sm \{ 0 \} : \vert \operatorname{arg}(z) \vert < \phi \}$$ denotes the open sector in $\C.$

\begin{theorem} \label{thm_elliptic_regularity}
There exists a constant $C_0> 0$ such that for
	any $\lambda \in \rho(A) \supset \overline{\Sigma_{\pi/2}} \sm \{0 \}$ and any $F\in H_0^2(U)\times L^2(U)$ the unique solution
	$u = (u_1, v_1) \in D(A)$ of
	\begin{equation}\label{3-1}
		(\lambda - A)u = F \in H^2_0(U) \times L^2(U)
	\end{equation}
	satisfies the estimate
	\begin{equation}
		\Vert u \Vert_{H^{2+\theta}(U) \times H^\theta(U)} \leq C_0 \Vert F \Vert_{H^\theta(U) \times H^{-2+\theta}(U)} \quad (\theta\in [0,2]).	\label{3-2}
	\end{equation}
In particular, for  $u = (u_1, v_1) \in D(A)$ solving
	\[
		(\lambda - A)u = \begin{pmatrix} 0 \\ f \end{pmatrix}
	\]
	with $f \in L^2(U)$ we obtain  the estimate
	\begin{equation}
		\Vert u_1 \Vert_{H^{2+\theta}(U)} 	 \leq C_0 \Vert f \Vert_{H^{-2+\theta}(U)} \quad (\theta\in[0,2]).			\label{3-3}
	\end{equation}
\end{theorem}

\begin{proof}
	By \cite{Chen-Triggiani89}, Proposition~3.1, $A$ is the generator of an analytic, exponentially stable, strongly continuous semigroup on $H^2_0(U) \times L^2(U).$
	Therefore, \eqref{3-1} is uniquely solvable, and we have the uniform resolvent estimate
\begin{equation}
  \label{3-4}
  \|u\|_{H^4(U)\times H^2(U)} \le C_1 \| F\|_{H^2(U)\times L^2(U)}
\end{equation}
with some constant $C_1>0$ independent of $F$ and $\lambda$.

	Let $A^\sharp := A'$ be the adjoint operator of $A$ and set
	\begin{align*}
		 E_0			& := H^2_0(U) \times L^2(U), \\
		 E_1			& := D(A) = \left(H^4(U) \cap H^2_0(U) \right) \times H^2_0(U), \\
		 E^\sharp_0	& := E_0' = H^{-2}(U) \times L^2(U), \\
		 E_1^\sharp		& := D(A^\sharp).
	\end{align*}
	Obviously, $E_0$ is reflexive and $E_1$ is dense in $E_0.$
	Since $A$ is the generator of an analytic $C_0$-semigroup on $E_0$ with domain $E_1$, in symbols $A \in \HH(E_1, E_0),$
	by \cite{Am95}, p. 13, Proposition 1.2.3,
	the same holds true for $A^\sharp$ on $E_0^\sharp$ with domain $E_1^\sharp,$ i.e. $A^\sharp \in \HH(E_1^\sharp, E_0^\sharp).$ \\
	Hence, we can define the interpolation-extrapolation
	scales $\{ (A_\alpha, E_\alpha) : \alpha \in \R \}$ and its dual scale $\{ (A^\sharp_\alpha, E^\sharp_\alpha) : \alpha \in \R \}.$ Then,
	Theorem 1.5.12 in \cite{Am95}
	 states that $E_\alpha$ is reflexive and we have
	\[
		(E_\alpha)' = E^\sharp_{-\alpha} \, \text{ and } \, (A_\alpha)' = A^\sharp_{-\alpha}
	\]
	for all $\alpha \in \R.$ Moreover, by \cite{Am88}, Theorem 6.1 and \cite{Am95}, Theorem 2.1.3 it holds that $A_\alpha$ and $A_\alpha^\sharp$ are
	generators of analytic $C_0$-semigroups in $E_\alpha$ with domain $E_{\alpha+1}$ and $E_\alpha^\sharp$ with domain $E_{\alpha+1}^\sharp$ for all $\alpha \in \R,$
	respectively. Again, we write $A_\alpha \in \HH(E_{\alpha+1}, E_\alpha)$ and $A_\alpha^\sharp \in \HH(E_{\alpha+1}^\sharp, E_\alpha^\sharp).$ \\
	In particular, $A_{-1}$ is the generator of an analytic $C_0$-semigroup on $E_{-1}$ with domain $E_0.$ By
	\cite{Am95}, Theorem 2.1.3,   $\lambda - A_{-1}$ is an isomorphism from $E_0$ to $E_{-1}$ and we have
	\begin{align*}
		\Vert (\mu - A_{-1})^{-1} \Vert_{L(E_{-1}, E_0)} \leq C \Vert (\mu - A)^{-1} \Vert_{L(E_0, E_1)} \leq C'
	\end{align*}
	for all $\mu \in \rho(A)$ with a constant $C'$ independent of $\mu.$ By Lemma \ref{3.2}, the space $E_{-1}$ equals
	\begin{align*}
		E_{-1} = \left( E_{-1} \right)'' = \left (E_1^\sharp \right)' = \left( D(A') \right)' = \left( L^2(U) \times H^2_0(U) \right)' = L^2(U) \times H^{-2}(U).
	\end{align*}
Therefore, there exists a constant $C_2>0$ such that
\begin{equation}
  \label{3-5}
  \|u\|_{H^2(U)\times L^2(U)} \le C_2 \|F\|_{L^2(U)\times H^{-2}(U)}.
\end{equation}
Now the inequality \eqref{3-2} follows by (real) interpolation between \eqref{3-4} and \eqref{3-5} with $C_0:= \max\{C_1,C_2\}$.

Considering the particular case $F=\binom{0}{f}$ and only the first component of $u$, we obtain \eqref{3-3}.
\end{proof}

%%%%%%%%%%%%%%%%%%%%%%%%%%%%%%%%%%%%%%%%%%%%%%%%%%%%%%%%%%%% SUBSECTION - Strong Stability, Analyticity? Lack thereof?%%%%%%%%%%%%%%%%%%%%%%%%%%%%%%%%%%%%%%%
\section{Exponential stability}\label{subsec:2_strong_exp_stability}
%%%%%%%%%%%%%%%%%%%%%%%%%%%%%%%%%%%%%%%%%%%%%%%%%%%%%%%%%%%%%%%%%%%%%%%%%%%%%%%%%%%%%%%%%%%%%%%%%%%%%%%%%%%%%%%%%%%%%%%%%%%%%%%%%%%%%%%%%%%%%%%%%%%%%%%%%%%%%%%%%

In this section, we continue the analysis of the coupled system \eqref{2_eq_tm_damped}--\eqref{2_eq_tm_undamped}. We will estimates the resolvent  $(\AA - \mathrm{i}\lambda)^{-1}$ of the corresponding first-order system  on the imaginary axis for $\lambda\in\R$ with $|\lambda|$ large. By a result due to Pr\"uss (\cite{Pr84}, Corollary 4), uniform boundedness of the resolvent on the imaginary axis implies exponential stability of the semigrroup.

We start with some identities which will be useful for our estimates. In the following, we will shortly write $x$ for the identity function $x\mapsto x$. For vectors $y,z\in\C^n$ we set $y\cdot z := \sum_{j=1}^n y_j z_j$ (note that this is not the scalar product in $\C^n$).

\begin{lemma}
  \label{4.1}
  Let $U\subset\R^n$ be a $C^4$-domain,  let $w\in H^4(U)$, and let $\nu\colon \partial U\to\R^n$ be the outer unit normal vector. Then,
  \begin{align*}
    2 \Re \int_U (x\cdot \nabla \bar w) \Delta^2w dx & = (4-n) \|\Delta w\|^2_{L^2(U)} + \int_{\partial U} (x\cdot \nu) |\Delta w|^2 dS  \\
    & + 2\Re \int_{\partial U} \Big[ (x\cdot \nabla \bar w)\partial_\nu\Delta w - \Delta w \partial_\nu(x\cdot \nabla\bar w)\Big] dS.
  \end{align*}
\end{lemma}

\begin{proof}
  This follows by straightforward calculation from the divergence theorem, applied to the vector field
  \[ V:= |\Delta w|^2 \, x + 2( x\cdot \nabla \bar w) \nabla\Delta w - 2\Delta w\nabla(x\cdot \nabla \bar w).\]
  Note that
  \[ \div V = n|\Delta w|^2 + x\cdot \nabla|\Delta w|^2 + 2(x\cdot\nabla \bar w)\Delta^2w - 2\Delta w \Delta(x\cdot\nabla\bar w)\]
  and $\Re(\div V) = 2\Re(x\cdot \nabla \bar w)\Delta ^2w+(n-4)|\Delta w|^2$. A more general variant of the statement is also known as Rellich's identity, see, e.g., \cite{mitidieri93}, Proposition~2.2, or \cite{Man13}, p.~238.
\end{proof}

\begin{lemma}
  \label{4.2}
  Let $U\subset\R^n$ be a $C^4$-domain, and let $w\in H^4(U)$ be a solution of $-\Delta^2 w+ \lambda^2w=z$ with $\lambda\in\R$ and $z\in L^2(U)$. Then we have
  \begin{align*}
    n\lambda^2\|w\|_{L^2(U)}^2 & + (4-n)\|\Delta w\|_{L^2(U)}^2 + \int_{\partial U} (x\cdot\nu) |\Delta w|^2 dS \\
     = & - 2\Re\int_U (x\cdot\nabla \bar w)z dx + \lambda^2 \int_{\partial U}(x\cdot\nu)|w|^2 dS \\
    & - 2 \Re \int_{\partial U}\Big[ (x\cdot\nabla\bar w)\partial_\nu\Delta w - \Delta w \partial_\nu(x\cdot\nabla\bar w)\Big] dS.
  \end{align*}
\end{lemma}

\begin{proof}
  Applying the divergence theorem to the vector field $|w|^2x$ and taking the real part, we obtain
  \[ 2 \Re\int_U (x\cdot\nabla\bar w)w\,dx = -n\|w\|^2_{L^2(U)} + \int_{\partial U} (x\cdot\nu) |w|^2 dS. \]
  From this and $\Delta^2 w = \lambda^2 w -z$ we get
  \begin{align*}
    2\Re\int_U (x\cdot\nabla \bar w)\Delta^2 w\,dx & = -2\Re\int_U (x\cdot\nabla\bar w)z\,dx -n\lambda^2\|w\|_{L^2(U)}^2 \\
    & \quad +\lambda^2\int_{\partial U} (x\cdot\nu) |w|^2 dS.
  \end{align*}
  Plugging this into the statement of Lemma~\ref{4.1}, the assertion follows.
\end{proof}

The following result can be found, e.g., in \cite{Man13}, Proof of Theorem~2.2.

\begin{lemma}
  \label{4.3}
  Let $U\subset\R^n$ be a $C^3$-domain, and let $S\subset\partial U$ be a nontrivial part of the boundary. Then, for every $w\in H^3(U)$ with $w=\partial_\nu w =0$ on $S$ we have $\partial_\nu(x\cdot\nabla w) = (x\cdot \nu) \Delta w$ on $S$.
\end{lemma}

In the next step, we consider the resolvent equation $(-\mathrm{i}\lambda + \AA) U = F$ for a particular right-hand side $F=(0,0,0,g_2)^\top$ with inhomogeneous transmission conditions. More precisely, we consider
\begin{align}
		-\mathrm{i} \lambda u_1 + v_1 								& = 0 \,\, \text{ in } \Omega_1, 	\label{eq_apriori_1}\\
		- \Delta^2 u_1 + \rho \Delta v_1 - \mathrm{i}\lambda v_1 	& = 0 \,\, \text{ in } \Omega_1, 	\label{eq_apriori_2}\\
		-\mathrm{i}\lambda u_2 + v_2 								& = 0 \,\, \text{ in } \Omega_2, 	\label{eq_apriori_3}\\
		- \Delta^2 u_2 - \mathrm{i}\lambda v_2						& = g_2 \text{ in } \Omega_2		\label{eq_apriori_4}
\end{align}
with transmission conditions
\begin{align}
		\left.
		\begin{aligned}
			\Delta u_1 	& = \Delta u_2, \\
			-\mathrm{i}\lambda \rho \partial_\nu u_1 + \partial_\nu \Delta u_1 & = \partial_\nu \Delta u_2 + \mathrm{i}\lambda \rho \partial_\nu w_1
		\end{aligned}
		\right\} \, \text{ on } \Gamma.
		\label{eq_apriori_tm}
\end{align}

The following a priori estimate will be the crucial step for the proof of exponential stability.

\begin{proposition} \label{prop_damped_undamped_apriori}
Let $w_1 \in H^4(\Omega_1)$ and $g_2 \in L^2(\Omega_2)$ be given. Then,
 there exists $\lambda_0 > 0$ and a constant $C > 0$ (only depending on $n, \rho, \delta_0$ and $\lambda_0$)
	such that for any solution $U = (u_1, u_2, v_1, v_2)^\top \in X(\Omega) \times X(\Omega)$ with $u_i \in H^4(\Omega_i)$ for $i = 1, 2$ of \eqref{eq_apriori_1}--\eqref{eq_apriori_tm}
	the estimate
	\begin{align*}
		\Vert U \Vert_\HH \leq C\left( \Vert g_2 \Vert_{L^2(\Omega_2)} + \vert \lambda \vert \Vert \partial_\nu w_1 \Vert_{L^2(\Gamma)} \right)
			 \quad (\lambda \in \R, \vert \lambda \vert > \lambda_0)
	\end{align*}
	holds.
\end{proposition}

In the following proof, we will use a generic constant $C$ independent of $\lambda, U,$ and $F$. Moreover, an estimate of the form $\|\cdot\| \le \epsilon \|\cdot\|_1 + C_\epsilon\|\cdot\|_2$ has to be understood in the sense that for every small $\epsilon>0$ there exists a constant $C_\epsilon>0$ such that the inequality holds. Again $C_\epsilon$ denotes a generic constant. Note that all constants may depend on $\rho$.

\begin{proof}
We have to estimate
\[ \|U\|_\HH = \Big( \|\Delta u_1\|_{L^2(\Omega_1)}^2 + \|\Delta u_2\|_{L^2(\Omega_2)}^2 + \|v_1\|_{L^2(\Omega_1)}^2 + \|v_2\|_{L^2(\Omega_2)}^2\Big)^{1/2}.\]
The proof is done in several steps.

\medskip

\textbf{(i)} \textit{Estimate of $v_1$.} Let $\lambda \in \R$ with $\vert \lambda \vert \gg 1$ and $U = (u_1, u_2, v_1, v_2)\in X(\Omega) \times X(\Omega)$ be a solution of \eqref{eq_apriori_1}--\eqref{eq_apriori_tm}.
	Hence, $(u_1, u_2)$ is a solution of
	\begin{align}
		-\Delta^2 u_1 + \mathrm{i} \lambda \rho \Delta u_1 + \lambda^2 u_1	& = 0, 												\label{eq_apriori_5} \\
		-\Delta^2 u_2 + \lambda^2 u_2										& = g_2 											\label{eq_apriori_6}
	\end{align}
	in $\Omega_1 \times \Omega_2$ satisfying the transmission conditions \eqref{eq_apriori_tm}. By the definition of $X(\Omega)$, we have  $u_1 = \partial_\nu u_1 = 0$ on $\Gamma_1$.
	In order to show the assertion of the theorem, we need to establish an estimate of the form
	\[
		\Vert U \Vert^2_\HH \leq \epsilon \Vert U \Vert^2_\HH + C_\epsilon
			\left( \Vert g_2 \Vert_{L^2(\Omega_2)}^2 + \vert \lambda \vert^2 \Vert \partial_\nu w_1 \Vert_{L^2(\Gamma)}^2 \right)
	\]
	Similar to the proof of the dissipativity of $\AA$ in Theorem \ref{2_thm_well_posedness_damped_undamped}, we obtain
	\begin{align*}
		\Re \langle F, U \rangle_\HH = \Re \langle \AA U, U \rangle_\HH = - \rho \Vert \nabla v_1 \Vert_{L^2(\Omega_1)}^2
			- \Re \int_\Gamma \mathrm{i}\lambda \rho \overline{v_1} \partial_\nu w_1 \, dS.
	\end{align*}
Therefore, Poincar\'{e} and Young's inequality yield
	\begin{align*}
		\Vert v_1 \Vert_{H^1(\Omega_1)}^2
			& \leq C  \left( \Vert g_2 \Vert_{L^2(\Omega_2)} \Vert U \Vert_\HH
				+ \vert \lambda \vert \Vert \partial_\nu w_1 \Vert_{L^2(\Gamma)} \Vert v_1 \Vert_{H^1(\Omega_1)} \right) 	\\
			& \leq \epsilon  \left( \Vert U \Vert_\HH^2 + \Vert v_1 \Vert_{H^1(\Omega_1)}^2 \right)
				+ C_{\epsilon } \left( \Vert g_2 \Vert_{L^2(\Omega_2)}^2 + \vert \lambda \vert^2 \Vert \partial_\nu w_1 \Vert_{L^2(\Gamma)}^2 \right),
	\end{align*}
	that is
	\begin{align}
		\Vert v_1 \Vert_{H^1(\Omega_1)}^2 \leq \epsilon  \Vert U \Vert_\HH^2
			+ C_{\epsilon } \left( \Vert g_2 \Vert_{L^2(\Omega_2)}^2 + \vert \lambda \vert^2 \Vert \partial_\nu w_1 \Vert_{L^2(\Gamma)}^2 \right).
		\label{eq_apriori_est_v1}
	\end{align}
	Together with \eqref{eq_apriori_1} this implies
	\begin{align}
			\Vert u_1 \Vert_{H^1(\Omega_1)}^2
		\leq \epsilon \vert \lambda \vert^{-2} \Vert U \Vert_\HH^2
			+ C_{\epsilon} \left( \vert \lambda \vert^{-2} \Vert g_2 \Vert_{L^2(\Omega_2)}^2 + \Vert \partial_\nu w_1 \Vert_{L^2(\Gamma)}^2 \right).
		\label{eq_apriori_est_u1}
	\end{align}

\medskip

\textbf{(ii)} \textit{Estimate of $\Delta u_1$ and $\Delta u_2$.}
We multiply \eqref{eq_apriori_5} by $-\overline{u_1}$ and \eqref{eq_apriori_6}
	by $-\overline{u_2}.$ Integration by parts and summing up yields
	\begin{align*}
		\Vert \Delta u_1 \Vert^2_{L^2(\Omega_1)} & + \Vert \Delta u_2 \Vert^2_{L^2(\Omega_2)} + \mathrm{i}\lambda \rho \Vert \nabla u_1 \Vert^2_{L^2(\Omega_2)} \\
			& = \lambda^2 \left( \Vert u_1 \Vert^2_{L^2(\Omega_1)} + \Vert u_2 \Vert^2_{L^2(\Omega_2)} \right)
				- \langle g_2, u_2 \rangle_{L^2(\Omega_2)} + \mathrm{i}\lambda \rho \int_\Gamma \overline{u_1} \partial_\nu w_1 \, dS,
	\end{align*}
	where we have used the transmission conditions \eqref{eq_apriori_tm} and $u_1 = \partial_\nu u_1 = 0$ on $\Gamma_1.$ Taking the real part and observing $v_j=i\lambda u_j$, we see that
	\begin{align}
		\quad \Vert \Delta u_1 \Vert^2_{L^2(\Omega_1)} & + \Vert \Delta u_2 \Vert^2_{L^2(\Omega_2)} \leq \vert \lambda \vert^2 \big( \Vert u_1 \Vert^2_{L^2(\Omega_1)} + \Vert u_2 \Vert^2_{L^2(\Omega_2)} \big) \nonumber\\
				& \quad	+ \Vert g_2 \Vert_{L^2(\Omega_2)} \Vert u_2 \Vert_{L^2(\Omega_2)}
					+ \vert \lambda \vert \rho \Vert u_1 \Vert_{L^2(\Gamma)} \Vert \partial_\nu w_1 \Vert_{L^2(\Gamma)} \nonumber\\
& =  \big( \Vert v_1 \Vert^2_{L^2(\Omega_1)} + \Vert v_2 \Vert^2_{L^2(\Omega_2)} \big) \nonumber\\
& \quad	+ \Vert g_2 \Vert_{L^2(\Omega_2)} \Vert u_2 \Vert_{L^2(\Omega_2)}
					+ \rho \Vert v_1 \Vert_{L^2(\Gamma)} \Vert \partial_\nu w_1 \Vert_{L^2(\Gamma)}.\label{4-1}
\end{align}
Assuming $|\lambda|\ge 1$, we get with the trace theorem and Young's inequality
\begin{align*}
  \|g_2\|_{L^2(\Omega_2)}\|u_2\|_{L^2(\Omega_2)} & \le \tfrac 12 \|g_2\|_{L^2(\Omega_1)}^2 + \tfrac12 \|v_2\|_{L^2(\Omega_2)}^2,\\
  \rho\|v_1\|_{L^2(\Gamma)}\|\partial_\nu w_1\|_{L^2(\Gamma)} & \le \tfrac\rho 2\|v_1\|_{H^1(\Omega_1)}^2 + \tfrac \rho2 \|\partial_\nu w_1\|_{L^2(\Gamma)}^2.
\end{align*}
Inserting this into \eqref{4-1} yields
\begin{align}
  \|\Delta u_1\|_{L^2(\Omega_1)}^2  & + \|\Delta u_2\|_{L^2(\Omega_2)}^2 \le (1+\tfrac\rho 2) \|v_1\|_{H^1(\Omega_1)}^2 + \|v_2\|_{L^2(\Omega_2)}^2 \nonumber\\
  & \qquad + \|g_2\|_{L^2(\Omega_1)}^2 + \tfrac \rho2 \|\partial_\nu w_1\|_{L^2(\Gamma)}^2 \nonumber\\
  & \le \epsilon  \Vert U \Vert_\HH^2 + C  \|v_2\|_{L^2(\Omega_2)}^2 \nonumber\\
  & \qquad + C_{\epsilon} \big( \Vert g_2 \Vert_{L^2(\Omega_2)}^2 + |\lambda|^2\,\Vert \partial_\nu w_1 \Vert_{L^2(\Gamma)}^2 \big).\label{4-7}
\end{align}
Here, in the last step we estimated $\|v_1\|_{H^1(\Omega_1)}$ due to \eqref{eq_apriori_est_v1}.

\medskip

\textbf{(iii)} \textit{Estimate of $v_2$.} We apply Lemma~\ref{4.2} in $U=\Omega_2$ with $w=u_2$ and $z=g_2$ and obtain, noting $i\lambda u_2=v_2$,
\begin{align}
  n\|v_2\|_{L^2(\Omega_2)}^2 & = -(4-n)\|\Delta u_2\|_{L^2(\Omega_2)}^2 - \int_\Gamma (x\cdot \nu)|\Delta u_2|^2 dS \nonumber\\
  & - 2\Re \int_{\Omega_2} (x\cdot\nabla \bar{u_2}) g_2\, dx + \lambda^2 \int_\Gamma (x\cdot\nu) |u_2|^2 dS \nonumber\\
  & - 2\Re\int_\Gamma\Big[ (x\cdot\nabla u_2)\partial_\nu\Delta\bar{u_2} - \Delta u_2 \partial_\nu (x\cdot\nabla\bar{u_2})\Big] dS\label{4-2}.
\end{align}
In the same way, we apply Lemma~\ref{4.2} in $U=\Omega_1$ with $w=u_1$ and $z=-\rho\Delta v_1$. Here we remark that $-\Delta^2u_1+\lambda^2u_1=-\rho\Delta v_1$ by \eqref{eq_apriori_1} and \eqref{eq_apriori_2}. Moreover, the normal vector $\nu$ is the outer normal at the part $\Gamma_1$ of the boundary $\partial\Omega_1$, but $\nu$ is the inner normal at the part $\Gamma$ of $\partial\Omega_1$. We obtain
\begin{align}
  n \|v_1\|_{L^2(\Omega_1)}^2 & = -(4-n)\|\Delta u_1\|_{L^2(\Omega_1)}^2 - \int_{\Gamma_1}(x\cdot\nu)|\Delta u_1|^2\,dx\nonumber\\
  & + \int_{\Gamma} (x\cdot\nu)|\Delta u_1|^2\,dx + 2\Re\int_{\Omega_1} (x\cdot\nabla\bar{u_1})\rho\Delta v_1\,dx\nonumber\\
  & + \lambda^2\int_{\Gamma_1} (x\cdot\nu)|u_1|^2 dS - \lambda^2\int_\Gamma(x\cdot\nu)|u_1|^2dS\nonumber\\
  & -2\Re\int_{\Gamma_1}\Big[ (x\cdot\nabla {u_1})\partial_\nu\Delta\bar{u_1} - \Delta u_1 \partial_\nu(x\cdot\nabla\bar{u_1})\Big]dS\nonumber\\
  & + 2\Re\int_{\Gamma}\Big[ (x\cdot\nabla {u_1})\partial_\nu\Delta\bar{u_1} - \Delta u_1 \partial_\nu(x\cdot\nabla\bar{u_1})\Big]dS .\label{4-3}
\end{align}
Due to the condition $(u_1,u_2)\in X(\Omega)$ and the transmission conditions \eqref{eq_apriori_tm}, we have
\begin{equation}
  \label{4-4}
  u_1=u_2,\; \nabla u_1=\nabla u_2,\; \Delta u_1 =\Delta u_2,\; \partial_\nu\Delta u_2 = \partial_\nu\Delta u_1 -i\lambda\rho\partial_\nu(u_1+w_1)\quad \text{on }\Gamma.
\end{equation}
Let $\tilde u_2\in H^4(\Omega)$ be a regular extension of $u_2$ to $\Omega$, and define $\varphi:=u_1-\tilde u_2|_{\Omega_1} \in H^4(\Omega_1)$. Then $\varphi=\partial_\nu\varphi=0$ on $\Gamma$, and an application of Lemma~\ref{4.3} yields
\[ \partial_\nu(x\cdot\nabla\varphi) = (x\cdot\nu)\Delta\varphi = (x\cdot \nu) (\Delta u_1-\Delta u_2) = 0\quad\text{on }\Gamma\]
which gives
\begin{equation}
  \label{4-5}
  \partial_\nu(x\cdot\nabla u_1) = \partial_\nu(x\cdot\nabla u_2)\quad \text{on }\Gamma.
\end{equation}
Moreover, with Lemma~\ref{4.3} again we get
\begin{equation}
  \label{4-6}
  u_1 = 0,\; \nabla u_1 = 0,\; \partial_\nu(x\cdot\nabla u_1)=(x\cdot\nu)\Delta u_1\quad\text{on }\Gamma_1.
\end{equation}
Adding \eqref{4-2} and \eqref{4-3} and taking into account \eqref{4-4}--\eqref{4-6}, we obtain
\begin{align*}
  n\big(\|v_1\|_{L^2(\Omega_1)}^2 & + \|v_2\|_{L^2(\Omega_2)}^2 \big) = -(4-n)\big(\|\Delta u_1\|_{L^2(\Omega_1)}^2 + \|\Delta u_2\|_{L^2(\Omega_2)}^2\big) \\
  & + 2\Re\Big[ i\lambda\rho\int_{\Omega_1} (x\cdot\nabla\bar{u_1})\Delta u_1\, dx\Big]
  - 2\Re\Big[ \int_{\Omega_2} (x\cdot\nabla\bar{u_2})g_2\,dx\Big]\\
  & + 2\Re\big[ i\lambda\rho\int_{\Gamma} (x\cdot \nabla u_1)\partial_\nu\overline{(u_1+w_1)}dS\Big]+ \int_{\Gamma_1} (x\cdot\nu) |\Delta u_1|^2 dS.
\end{align*}
Therefore,
\begin{align}
  \|v_2\|_{L^2(\Omega_2)}^2 & \le C \Big[ |\lambda|\,\|u_1\|_{H^1(\Omega_1)}\|\Delta u_1\|_{L^2(\Omega_1)} + \|\nabla u_2\|_{L^2(\Omega_2)} \|g_2\|_{L^2(\Omega_2)} \nonumber\\
  & + |\lambda|\,\|u_1\|_{H^1(\Gamma)}^2 + |\lambda|\, \|u_1\|_{H^1(\Gamma)}\|\partial_\nu w_1\|_{L^2(\Gamma)} + \|u_1\|_{H^2(\Gamma_1)}^2\Big]\label{4-8}.
\end{align}
We estimate the first four terms on the right-hand side of \eqref{4-8} while the last term will be treated in part (iv) of this proof.

The first term in \eqref{4-8} can be estimated by \eqref{eq_apriori_est_v1}, $\|\Delta u_1\|_{L^2(\Omega_1)}\le \|U\|_\HH$ and Young's inequality. We obtain
\begin{align}
  |\lambda|\,\|u_1\|_{H^1(\Omega_1)}\|\Delta u_1\|_{L^2(\Omega_1)} & \le \Big( \epsilon\|U\|_\HH + C_\epsilon\big(\|g_2\|_{L^2(\Omega_2)}+ |\lambda|\,\|\partial_\nu w_1\|_{L^2(\Gamma)}\big)\Big)\|U\|_\HH \nonumber\\
  & \le \epsilon \|U\|_\HH^2 + C_\epsilon \big(\|g_2\|_{L^2(\Omega_2)}+ |\lambda|\,\|\partial_\nu w_1\|_{L^2(\Gamma)}\big)^2.\label{4-9}
\end{align}

For the second term in \eqref{4-8}, we apply Green's formula, using $u_1=u_2$ and $\partial_\nu u_1=\partial_\nu u_2$ on $\Gamma$ to see that
\begin{align*}
  \|\nabla u_2\|_{L^2(\Omega_2)}^2 & \le \|u_2\|_{L^2(\Omega_2)}\|\Delta u_2\|_{L^2(\Omega_2)} + \|u_1\|_{L^2(\Gamma)} \|\nabla u_1\|_{L^2(\Gamma)} \\
  & \le \tfrac12 \|u_2\|_{L^2(\Omega_2)}^2 + \tfrac12 \|\Delta u_2\|_{L^2(\Omega_2)}^2 + C \|u_1\|_{H^2(\Omega_1)}^2 \\
  & \le C\big( \|u_1\|_{H^2(\Omega_1)}^2 + \|u_2\|_{H^2(\Omega_2)}^2\big) \le C\|U\|_\HH^2.
\end{align*}
For the last inequality, we have applied Remark~\ref{2.1} b). Therefore, the second term in \eqref{4-8} can be estimated by
\[ \|\nabla u_2\|_{L^2(\Omega_2)} \|g_2\|_{L^2(\Omega_2)} \le \epsilon\|U\|_\HH^2 + C_\epsilon \|g_2\|_{L^2(\Omega_2)}^2.\]

For the third term in \eqref{4-8} we use interpolation to see that
\begin{align}
  |\lambda|\,\|u_1\|_{H^1(\Gamma)}^2 & \le C|\lambda| \,\|u_1\|_{H^{3/2}(\Omega_1)}^2 \le C |\lambda|\,\|u_1\|_{H^1(\Omega_1)} \|u_1\|_{H^2(\Omega_1)} \nonumber\\
  & \le\epsilon \|U\|_\HH^2 + C_\epsilon\big(\|g_2\|_{L^2(\Omega_2)}^2+ |\lambda|^2\,\|\partial_\nu w_1\|_{L^2(\Gamma)}^2\big)\label{4-10}.
\end{align}

Similarly, for the fourth term in \eqref{4-8} we write
\[ |\lambda|\,\|u_1\|_{H^1(\Gamma)}\|\partial_\nu w_1\|_{L^2(\Gamma)} \le \tfrac12 \|u_1\|_{H^1(\Gamma)}^2 + \tfrac12|\lambda|^2 \|\partial_\nu w_1\|_{L^2(\Gamma)}^2.\]
As we have $|\lambda|\ge 1$, this again can be estimated by the right-hand side of \eqref{4-10}. Altogether, we obtain
\begin{equation}
  \label{4-11}
  \|v_2\|_{L^2(\Omega_2)}^2 \le \epsilon \|U\|_\HH^2 + C_\epsilon\big(\|g_2\|_{L^2(\Omega_2)}^2+ |\lambda|^2\,\|\partial_\nu w_1\|_{L^2(\Gamma)}^2\big) + C \|u_1\|_{H^2(\Gamma_1)}^2.
\end{equation}

\medskip

\textbf{(iv)} \textit{Estimate of $u_1$ on $\Gamma_1$.}
The only term still left is $\|u_1\|_{H^2(\Gamma_1)}$.
We introduce a cut-off function
	$\chi \in C^\infty(\overline{\Omega_1}), 0 \leq \chi \leq 1,$ satisfying $\chi = 1$ in a neighbourhood of $\Gamma_1$ and $\chi = 0$ in a neighbourhood of
	the transmission interface $\Gamma.$ Now, set
	\[
		z_1 := \chi u_1, \quad z_2 := \mathrm{i} \lambda z_1, \quad z := (z_1, z_2)^\top.
	\]
	Then, since $u_1$ is a solution of \eqref{eq_apriori_5}, $z$ satisfies
	\begin{align*}
		\left( -\mathrm{i}\lambda + A \right)z = \begin{pmatrix} 0 \\ \widetilde f \end{pmatrix},
	\end{align*}
	where $A$ is defined as in \eqref{3-6} and
	\begin{align*}
		\widetilde f
			& = (-\Delta^2\chi)u_1 - 2(\nabla \Delta \chi)\cdot \nabla u_1 - \Delta\chi \Delta u_1 - 2\Delta(\nabla \chi \cdot \nabla u_1) \\
			& \quad - \Delta \chi \Delta u_1 - 2 \nabla \chi \cdot \nabla \Delta u_1 + \mathrm{i}\lambda \rho \left( (\Delta \chi) u_1 + 2 \nabla\chi \cdot \nabla u_1 \right)
			\in L^2(\Omega_1) \\
& =  B_3(D, \chi)u_1 + \mathrm{i}\lambda \rho \left( (\Delta \chi)u_1 + 2 \nabla \chi \cdot \nabla u_1 \right)
	\end{align*}
	with a $\lambda$-independent differential operator $B_3(D, \chi)$ of order $3$ with coefficients only consisting of derivatives of the $C^\infty$-function $\chi.$ Hence,
	$$B_3(D, \chi) \in L(H^{3/2}(\Omega_1), H^{-3/2}(\Omega_1)).$$	
	From Theorem~\ref{thm_elliptic_regularity} with $\theta=\frac12$, we obtain
	\begin{align*}
		\|u_1\|_{H^2(\Gamma_1)}  & = \|z_1\|_{H^2(\Gamma_1)} \le C \Vert z_1 \Vert_{H^{5/2}(\Omega_1)}
			\leq C \Vert \widetilde f \Vert_{H^{-3/2}(\Omega_1)}  \\
			& \leq C  \left( \Vert B_3(D, \chi) u_1 \Vert_{H^{-3/2}(\Omega_1)} + \vert \lambda \vert \Vert u_1 \Vert_{H^1(\Omega_1)} \right)  \\
			& \leq C  \left( \Vert u_1 \Vert_{H^{3/2}(\Omega_1)} + \vert \lambda \vert \Vert u_1 \Vert_{H^1(\Omega_1)} \right) \\
& \le C \big( |\lambda|^{1/2} \|u_1\|_{H^{3/2}(\Omega_1)} + \|v_1\|_{H^1(\Omega_1)}\big).
	\end{align*}
Now, \eqref{eq_apriori_est_v1} and  \eqref{4-10} yield
\begin{equation}
 \label{4-12}
 \|u_1\|_{H^2(\Gamma_1)}^2 \le  \epsilon \|U\|_\HH^2 + C_\epsilon\big(\|g_2\|_{L^2(\Omega_2)}^2+ |\lambda|^2\,\|\partial_\nu w_1\|_{L^2(\Gamma)}^2\big).
\end{equation}

\medskip

The assertion of the Proposition now follows from
\eqref{eq_apriori_est_v1}, \eqref{4-7}, \eqref{4-11}, and \eqref{4-12}.
\end{proof}

\begin{theorem}\label{thm_exp_stability_structurally_damped_undamped}
There exists a constant $C = C(\rho) > 0$ such that
	\[
		\Vert (-\mathrm{i}\lambda + \AA)^{-1} \Vert_{L(\HH)} \leq C \quad (\lambda \in \R \sm \{ 0 \}, \vert \lambda \vert > \lambda_0)
	\]
	for some $\lambda_0 > 0.$
	Consequently, the $C_0$-semigroup $(\TT(t))_{t\geq 0}$ generated by $\AA$ is exponentially stable, i.e. there exist constants $M > 0$ and $\kappa > 0$ such that
	\[
		\Vert \TT(t)U^0 \Vert_\HH \leq Me^{-\kappa t} \Vert U^0 \Vert_\HH \quad (t \geq 0)
	\]
	holds for all $U^0 \in \HH.$
\end{theorem}
\begin{proof}
	Let $\lambda \in \R$ with $\vert \lambda \vert \gg 1$ and let $F = (f_1, f_2, g_1, g_2) \in \HH.$
	Furthermore, let $U = (u_1, u_2, v_1, v_2) \in D(\AA)$ be the unique solution
	of
	\[
		(-\mathrm{i}\lambda + \AA)U = F,
	\]
	i.e. $U$ satisfies
	\begin{align*}
		-\mathrm{i} \lambda u_1 + v_1 								& = f_1 \,\, \text{ in } \Omega_1, 	\\
		- \Delta^2 u_1 + \rho \Delta v_1 - \mathrm{i}\lambda v_1 	& = f_2 \,\, \text{ in } \Omega_1, 	\\
		-\mathrm{i}\lambda u_2 + v_2 								& = g_1 \,\, \text{ in } \Omega_2, 	\\
		- \Delta^2 u_2 - \mathrm{i}\lambda v_2						& = g_2 \, \, \text{ in } \Omega_2	.	
	\end{align*}
	In order to show the assertion, we will subtract the solution $W$ of a structurally damped plate equation with clamped boundary conditions on the whole domain $\Omega$ from $U.$
	For this difference we will be able to use the a-priori estimate from Proposition \ref{prop_damped_undamped_apriori}, whereas for $W$ an appropriate estimate is known. \\
	Recall the definition of the operator $A$ from \eqref{3-6} and define
	\[
		\widetilde W = (w, z) \in D(A) = \left( H^4(\Omega) \cap H^2_0(\Omega) \right) \times H^2_0(\Omega)
	\]
	by
	\begin{align*}
		\widetilde W := (-\mathrm{i}\lambda + A)^{-1} \begin{pmatrix} \chi_1 f_1 + \chi_2 f_2 \\ \chi_1 g_1 + \chi_2 g_2 \end{pmatrix},
	\end{align*}
	where $\chi_i$ is the characteristic function on $\Omega_i$ for $i = 1, 2.$ Since $\chi_1 f_1 + \chi_2 f_2 \in H^2_0(\Omega)$ and
	$\chi_1 g_1 + \chi_2 g_2 \in L^2(\Omega)$ due to the definition of $\HH,$ by Theorem \ref{thm_elliptic_regularity}, $\widetilde W$ is well-defined.
	In the following, we denote the restrictions of the components of $ \widetilde W$ by $w_i := w\vert_{\Omega_i}$ and $z_i := z\vert_{\Omega_i}$ for $i=1,2.$
	Finally, we set
	\begin{align*}
		W := (w_1, w_2, z_1, z_2) \in X(\Omega) \times X(\Omega).
	\end{align*}
	Note that $u_i \in H^4(\Omega_i)$ for $i = 1, 2.$
	With this definitions, we obtain that the difference $U-W$ satisfies
	\begin{align}
		(-\mathrm{i}\lambda + \AA) (U - W)
			& = F - (-\mathrm{i}\lambda + \AA) W
			 = (0,0,0,\widetilde g_2)^\top					 \label{eq_diff_U_W}																								
	\end{align}
	with $\widetilde g_2 = \rho \Delta z_2 \in L^2(\Omega_2),$ subject to the transmission conditions
	\begin{align*}
		\left\{
		\begin{aligned}
			\Delta (u_1 - w_1) & = \Delta (u_2 - w_2), \\
			-\mathrm{i}\lambda \rho \partial_\nu (u_1 - w_1) + \partial_\nu \Delta (u_1 - w_1) & = \partial_\nu \Delta (u_2 - w_2) + \mathrm{i}\lambda \rho \partial_\nu w_1.
		\end{aligned}
		\right.
	\end{align*}
	Thanks to Proposition \ref{prop_damped_undamped_apriori}, we have
\begin{equation}
\label{4-13}
		\Vert U - W\Vert_\HH \leq C\left( \Vert \widetilde g_2 \Vert_{L^2(\Omega_2)} + \vert \lambda \vert \Vert \partial_\nu w_1 \Vert_{L^2(\Gamma)} \right).
\end{equation}
An application of Theorem~\ref{thm_elliptic_regularity} with $U=\Omega$ and $\theta=2$ gives
	\begin{align*}
		\Vert \widetilde g_2 \Vert_{L^2(\Omega_2)} & = \rho \Vert \Delta z_2 \Vert_{L^2(\Omega_2)} \leq C \| z_2\|_{H^2(\Omega_2)} \le C \|\tilde W\|_{H^4(\Omega)\times H^2(\Omega)} \\
& \le C \Big\| \begin{pmatrix} \chi_1 f_1 + \chi_2 f_2 \\ \chi_1 g_1 + \chi_2 g_2 \end{pmatrix} \Big\|_{H^2(\Omega)\times L^2(\Omega)} \le C \|F\|_\HH.
	\end{align*}
Since $A$ is the generator of a bounded, analytic $C_0$-semigroup on $H^2_0(\Omega) \times L^2(\Omega)$ by Theorem \ref{thm_elliptic_regularity}, we see that
\[ \vert \lambda \vert \Vert \partial_\nu w_1 \Vert_{L^2(\Gamma)} \le C \vert \lambda \vert \Vert w_1 \Vert_{H^2(\Omega_1)} \leq C\Vert F \Vert_\HH.\]
Therefore, \eqref{4-13} yields
	\begin{align*}
		\Vert U - W \Vert_\HH \leq C \Vert F \Vert_\HH.
	\end{align*}
	Invoking Theorem \ref{thm_elliptic_regularity} again, we deduce
	\begin{align*}
		\Vert U \Vert_\HH \leq \Vert U - W \Vert_\HH + \Vert W \Vert_\HH \leq C\Vert F \Vert_\HH
	\end{align*}
	with a constant $C = C(\rho) > 0.$ This proves the theorem.
\end{proof}

\bibliographystyle{abbrv}
%\nocite{*}

\end{document}